\documentclass[11pt]{amsart}

\usepackage{epsfig,overpic,comment}
\usepackage[usenames,dvipsnames,svgnames,table]{xcolor}
\usepackage[hyphens]{url}
\usepackage[pagebackref,linktocpage=true,colorlinks=true,linkcolor=Blue,citecolor=BrickRed,urlcolor=RoyalBlue]{hyperref}
\usepackage[msc-links,abbrev]{amsrefs}
\usepackage{amsmath,amsthm,amssymb,marginnote}
\usepackage{enumerate}
\usepackage[T1]{fontenc}
\usepackage{cancel}

\newtheorem{theorem}{Theorem}[section]
\newtheorem{corollary}[theorem]{Corollary}
\newtheorem{lemma}[theorem]{Lemma}
\newtheorem{proposition}[theorem]{Proposition}

\theoremstyle{definition}
\newtheorem{note}[theorem]{Note}

\newcommand{\be}{\begin{equation}}
\newcommand{\ee}{\end{equation}}
\newcommand{\ol}{\overline}
\newcommand{\ul}{\underline}
\newcommand{\R}{\mathbf{R}}
\newcommand{\C}{\mathcal{C}}
\newcommand{\G}{\mathcal{G}}
\newcommand{\M}{\mathcal{M}}
\newcommand{\X}{\mathcal{X}}
\newcommand{\ff}{\mathrm{I\!I}}
\newcommand{\CAT}{\textup{CAT}}

\renewcommand{\S}{\mathbf{S}}
\renewcommand{\epsilon}{\varepsilon}
\renewcommand{\tilde}{\widetilde}
\renewcommand{\leq}{\leqslant}
\renewcommand{\geq}{\geqslant}

\DeclareMathOperator{\inte}{int}

\makeatletter
\DeclareFontFamily{U}{tipa}{}
\DeclareFontShape{U}{tipa}{m}{n}{<->tipa10}{}
\newcommand{\arc@char}{{\usefont{U}{tipa}{m}{n}\symbol{62}}}%

\newcommand{\arc}[1]{\mathpalette\arc@arc{#1}}

\newcommand{\arc@arc}[2]{%
  \sbox0{$\m@th#1#2$}%
  \vbox{
    \hbox{\resizebox{\wd0}{\height}{\arc@char}}
    \nointerlineskip
    \box0
  }%
}
\makeatother

\begin{document}
\setlength{\baselineskip}{1.2\baselineskip}

\title[Convexity and rigidity of hypersurfaces] 
{Convexity and rigidity of hypersurfaces\\ in Cartan-Hadamard manifolds}

\author{Mohammad Ghomi}
\address{School of Mathematics, Georgia Institute of Technology,
Atlanta, GA 30332}
\email{ghomi@math.gatech.edu}
\urladdr{www.math.gatech.edu/~ghomi}

\begin{abstract}
We show that in Cartan-Hadamard manifolds $M^n$, $n\geq 3$, closed infinitesimally convex hypersurfaces $\Gamma$  bound convex flat regions if  curvature of  $M^n$ vanishes on tangent planes of $\Gamma$. This encompasses Chern-Lashof-Sacksteder characterization of compact convex hypersurfaces in Euclidean space, and some results of Greene-Wu-Gromov on rigidity of Cartan-Hadamard manifolds. It follows that closed  simply connected surfaces  in $M^3$ with minimal total absolute curvature bound Euclidean convex bodies, as stated by Gromov in 1985. The proofs employ the Gauss-Codazzi equations, a generalization of Schur comparison theorem to $\CAT(k)$ spaces, and other techniques from Alexandrov geometry outlined by  Petrunin.
\end{abstract}

\date{\today \,(Last Typeset)}
\thanks{The author was supported by NSF grant DMS-2202337.
}

\maketitle


\section{Introduction}

A $\CAT^n(k_{\leq 0})$ manifold $M$ is a metrically complete simply connected Riemannian $n$-space with curvature $K_M\leq k\leq 0$, and \emph{locally convex} boundary $\partial M$. The last condition means that, when $\partial M\neq\emptyset$,  the second fundamental form of $\partial M$ is positive semidefinite with respect to the outward normal. When $\partial M=\emptyset$, $M$ is known as  a \emph{Cartan-Hadamard} manifold. A subset of $M$  is \emph{convex} if it contains the geodesic connecting every pair of its points, and is  called a \emph{convex body} if it also has nonempty interior.  A hypersurface $\Gamma$  in $M$ is \emph{convex} if it  bounds a convex body. We say $\Gamma$ is   \emph{infinitesimally convex}  if  its  principal curvatures do not assume opposite signs at any point, or its sectional curvatures $K_\Gamma\geq K_M$ on every tangent plane. Chern-Lashof-Sacksteder \cites{chern-lashof:tight1, chern-lashof:tight2,sacksteder1960} and do Carmo-Warner \cite{docarmo-warner1970}  showed, respectively, that infinitesimally convex closed hypersurfaces immersed in Euclidean space $\R^n$ or hyperbolic space $\textbf{H}^n$, $n\geq 3$, are convex. We extend these results to $\CAT^n(k_{\leq 0})$ manifolds. A region $X$ of $M$  is  \emph{$k$-flat} if $K_M\equiv k$ on $X$. 

\begin{theorem}\label{thm:main}
Let $\Gamma$ be a closed  infinitesimally convex $\C^{n}$  hypersurface immersed in a $\CAT^n(k_{\leq 0})$ manifold $M$, $n\geq 3$. Suppose that $K_{M}\equiv k$ on  tangent planes of $\Gamma$. Then $\Gamma$  bounds a $k$-flat convex body. In particular $\Gamma$ is an embedded sphere.
\end{theorem}

If $K_M\equiv k$ outside a compact set  $X$ in  $M$, and $\partial M=\emptyset$, then letting $\Gamma$ in the above theorem be a sphere enclosing $X$ yields that $K_M\equiv k$ everywhere. Thus Theorem \ref{thm:main} also extends some of the ``gap theorems'' \cite{schroeder-ziller1990,seshadri2009} first obtained by Greene-Wu \cite{greene-wu1982} and Gromov \cite{ballmann-gromov-schroeder}*{Sec. 5}. 
In the case where $n=3$ and $\Gamma$ is \emph{strictly convex}, i.e., the second fundamental form $\ff_\Gamma$ is  positive definite, the above result 
was established  by the author and Spruck \cite{ghomi-spruck2023}, generalizing earlier work of Schroeder-Strake \cite{schroeder-strake1989a} for $k=0$. Kleiner \cite[p. 42]{kleiner1992} had also observed a version of the above theorem when $n=3$ and $\Gamma$ has constant mean curvature, via a result of Schroeder-Ziller \cite[Thm. 7]{schroeder-ziller1990} which again requires strict convexity and also negative curvature.
Next result is an intrinsic version of Theorem \ref{thm:main}. 
Let $\M^n_k$ be the \emph{model space}, or complete simply connected $n$-manifold, of constant curvature $k\leq 0$.

\begin{theorem}\label{thm:main2}
Let $M^n$, $n\geq 3$, be a compact simply connected manifold with  infinitesimally convex $\C^{n}$ boundary $\Gamma$, and curvature $K_{M}\leq k\leq 0$ with $K_{M}\equiv k$ on  tangent planes of $\Gamma$. Suppose that each component  of $\Gamma$ is simply connected and contains a point where a principal curvature with respect to the outward normal is positive. Then $M$ is isometric to a convex body in $\M^n_k$. In particular $M$ is homeomorphic to a ball.
\end{theorem}

Conditions in the second sentence of the last theorem ensure that $M$ is not a tubular neighborhood of a closed geodesic, or the complement of a small open ball in a compact space form; see \cite{hass1994} for other examples of nonpositively curved manifolds with concave boundary.   Theorem \ref{thm:main} has the following application.
Let $GK:=\det(\ff_\Gamma)$ denote the \emph{Gauss-Kronecker} curvature of a surface $\Gamma$ in a Riemannian $3$-manifold. The \emph{total curvature} and \emph{total absolute curvature} of $\Gamma$ are given respectively by
$$
\mathcal{G}(\Gamma):=\int_\Gamma GK,\quad\quad\text{and}\quad\quad\tilde{\mathcal{G}}(\Gamma):=\int_\Gamma |GK|.
$$
Hypersurfaces which minimize $\tilde{\mathcal{G}}$ in a topological class are called \emph{tight} \cite{cecil-chern1985}, and have been studied extensively in $\R^n$ since Alexandrov \cite{alexandrov1938}. Let $|\Gamma|$ denote the area of $\Gamma$. 

\begin{corollary}\label{cor:main2}
Let $\Gamma$ be a closed simply connected  $\C^3$ surface immersed in a $\CAT^3(k_{\leq 0})$ manifold. Then
\be\label{eq:G2}
\tilde\G(\Gamma)\geq 4\pi -k|\Gamma|,
\ee
with equality only if $\Gamma$ bounds a $k$-flat convex body.
\end{corollary}
\begin{proof}
Let $M$ be the ambient space, and $K_\Gamma$ denote the sectional curvature of $\Gamma$.
By Gauss' equation, at every point $p\in \Gamma$,
\begin{equation}\label{eq:gauss}
GK(p)=K_{\Gamma}(p)-K_M (T_p \Gamma)\geq K_{\Gamma}(p)-k,
\end{equation}
where $T_p\Gamma$ is the tangent plane of $\Gamma$ at  $p$.
 Since $\Gamma$ is simply connected, $\int_\Gamma K_\Gamma=4\pi$ by Gauss-Bonnet theorem. Thus 
\begin{equation}\label{eq:tildeGG}
\tilde\G(\Gamma)\geq \G(\Gamma)=4\pi- \int_\Gamma K_M (T_p\Gamma)\geq 4\pi -k|\Gamma|.
\end{equation}
Equality in \eqref{eq:G2} forces equalities in \eqref{eq:tildeGG}.
In particular $\tilde{\mathcal{G}}(\Gamma)= \mathcal{G}(\Gamma)$, which yields  $GK\geq 0$ everywhere. So $\Gamma$ is infinitesimally convex. Furthermore $\int_\Gamma K_M (T_p\Gamma)= k|\Gamma|$, which yields $K_M (T_p\Gamma)=k$  for all $p\in\Gamma$. Now Theorem \ref{thm:main} completes the proof. 
\end{proof}

For $\Gamma$ strictly convex, the last result was established in \cite{ghomi-spruck2023}*{Cor. 1.2}. 
For surfaces in $\mathbf{H}^3$, the weaker inequality $\tilde\G(\Gamma)\geq 4\pi +|\Gamma_0|$, where $\Gamma_0$ denotes the boundary of the convex hull of $\Gamma$, 
had been known earlier \cite{langevin-solanes2003}*{Prop. 2}.  For surfaces in $\R^3$, Corollary \ref{cor:main2} dates back to Chern-Lashof \cites{chern-lashof:tight1, chern-lashof:tight2}, who showed that $\tilde\G(\Gamma)\geq 2\pi(2+2g)$, where $g$ is the topological genus of $\Gamma$. In 1966 Willmore-Saleemi \cite{willmore-saleemi} conjectured that the Chern-Lashof inequality holds in $\CAT^3(0)$ manifolds;  however, Solanes \cite{solanes2007} constructed closed surfaces $\Gamma$ in $\mathbf{H}^3$ of every genus $g\geq 1$   with $\tilde{\mathcal{G}}(\Gamma)\approx 8\pi$. In these examples $|\Gamma|\approx2\pi(2g+2)$, which shows that \eqref{eq:G2} does not hold for $g\geq 1$. So  Corollary \ref{cor:main2} is topologically sharp.

In 1985 Gromov \cite{ballmann-gromov-schroeder}*{p. 66 (b)} proposed that for all closed surfaces $\Gamma$ in a $\CAT^3(0)$ manifold, $\tilde\G(\Gamma)\geq 4\pi$ with equality only if $\Gamma$ bounds a $0$-flat convex body. Corollary \ref{cor:main2} settles this problem for $g=0$. For $g\geq 1$, we  show in Section \ref{sec:higher} that the inequality $\tilde\G(\Gamma)\geq 4\pi$ still holds; however, we cannot prove that $\Gamma$ is convex when equality holds. 

Proof of  Theorem \ref{thm:main} follows an approach suggested by Petrunin \cite{petrunin2022}. We first use the Gauss-Codazzi equations in Section \ref{sec:immersion} to show that $\Gamma$ is isometric to a hypersurface $\Gamma'$ in $\M^n_k$ with the same second fundamental form. It  follows from characterizations of convex hypersurfaces by Sacksteder \cite{sacksteder1960},  and Alexander \cite{alexander1977} that $\Gamma$ and $\Gamma'$ are both convex. Next in Section \ref{sec:schur} we  generalize Schur's comparison theorem to  $\CAT^n(k_{\leq 0})$ manifolds via Reshetnyak’s majorization theorem \cite{reshetnyak1968-english}. This result  is used to show in Section \ref{sec:proof} that the isometry $\Gamma\to\Gamma'$ preserves extrinsic distances. It follows  from the generalization of Kirszbraun's extension theorem by Lang-Schroeder \cite{lang-schroeder1997} that the mapping $\Gamma\to\Gamma'$ extends to an isometry of the convex bodies bounded by these hypersurfaces.  Theorem \ref{thm:main2} is proved similarly.

\section{Immersion into  Model Spaces}\label{sec:immersion}

Here we use the fundamental theorem of Riemannian hypersurfaces \cites{spivak:v4,dajczer1990} to immerse $\Gamma$ in Theorems \ref{thm:main} and \ref{thm:main2}  into the model space $\M^n_k$.
Let $M^n$ be a Riemannian $n$-manifold with connection $\nabla$ and metric $\langle\cdot,\cdot\rangle$.
The curvature operator of $M$ is given by
$
R(X,Y)Z:=\nabla_X\nabla_Y Z-\nabla_Y\nabla_X Z-\nabla_{[X,Y]}Z,
$
for vector fields $X$, $Y$, $Z$ on $M$. The sectional curvature of $M$ with respect to a plane $\sigma\subset T_p M$ is defined as
$$
K(\sigma)=K(X,Y):=\frac{\langle R(X,Y)Y,X\rangle}{|X\times Y|^2},
$$
where $X$, $Y$  span $\sigma$,  and $|X\times Y|:=(\langle X,X\rangle\langle Y,Y\rangle-\langle X,Y\rangle^2)^{1/2}$.  Let $\Gamma$ be a $\C^2$ immersed hypersurface in $M$.  The \emph{shape operator} and the \emph{second fundamental form} of $\Gamma$ with respect to a (continuous) unit normal vector field $N$ are given by
$$
A(X):=-\nabla_X(N),\quad\quad\text{and}\quad\quad \ff_\Gamma(X,Y):=\langle A(X),Y\rangle,
$$
respectively, for tangent vector fields $X$, $Y$ on $\Gamma$. The \emph{principal curvatures} of $\Gamma$ with respect to $N$ are eigenvalues of $A$. Let $M'$ be another Riemannian $n$-manifold, $f\colon\Gamma\to M'$ be an immersion, and set $\Gamma':=f(\Gamma)$. We say $f$ is \emph{isometric}, or $\Gamma\overset{f}{\to}\Gamma'$ is an isometry, if $\langle X, Y \rangle_M= \langle df(X), df(Y) \rangle_{M'}$; furthermore, $f$ \emph{preserves} $\ff_\Gamma$,  or $\Gamma$ and $\Gamma'$ have \emph{the same second fundamental form}, if  $\ff_{\Gamma}(X,Y)=\ff_{\Gamma'}(df(X), df(Y))$,  with respect to some normal vector fields. 

\begin{proposition}\label{prop:embedding}
Let $\Gamma$ be a simply connected $\C^{\alpha\geq 3}$  hypersurface immersed in  a Riemannian manifold $M^n$, $n\geq 3$. Suppose that for all points $p\in \Gamma$ and planes $\sigma\subset T_p M$, $K_M(\sigma)\leq k\leq 0$  with $K_{M}(\sigma)= k$ if $\sigma\subset T_p\Gamma$. Then there exists a $\C^\alpha$ isometric immersion $\Gamma\to \M^n_k$ which preserves $\ff_\Gamma$.
\end{proposition}

We always assume that $k\leq 0$ in this work. First we need to show:

\begin{lemma}\label{lem:RXY}
Let $p\in M$ be a point such that $K(\sigma)\leq k$  for all planes $\sigma\subset T_p M$. Suppose that there exists a hyperplane $H\subset T_p M$ such that $K(\sigma)= k$ for all planes $\sigma\subset H$.
 Then for every pair of vectors $X$, $Y\in H$, and orthogonal vector $N$ to $H$, $R(X,Y)N=0$.
\end{lemma}
\begin{proof}
It is enough to check that $\langle R(X,Y)N,Z\rangle=0$ for every vector $Z\in H$, since $\langle R(X,Y)N,N\rangle=0$.
Let $X_t:=X+t N$ and $\sigma_t$ be the plane spanned by $X_t$ and $Y$. Then
 $$
\langle R(X_t,Y)Y,X_t\rangle=K(\sigma_t) |X_t\times Y|^2 \leq k |X\times Y|^2 = \langle R(X,Y)Y,X\rangle.
$$ 
 So $t=0$ is a critical point of $t\mapsto \langle R(X_t,Y)Y,X_t\rangle$,
  which yields
\begin{align*}
\langle R(X,Y)Y,N\rangle&=\frac{1}{2}\frac{d}{d t}\Big|_{t=0}\langle R(X_t,Y)Y,X_t\rangle=0.
\end{align*}
It follows that
\begin{align*}\label{eq:R}
0=\langle R(X,Y+Z)(Y+Z),N\rangle
=\langle R(X,Y)Z,N\rangle+\langle R(X,Z)Y,N\rangle.
\end{align*}
So
$\langle R(X,Y)Z,N\rangle=\langle R(Z,X)Y,N\rangle$, which yields
$\langle R(Z,X)Y,N\rangle=\langle R(Y,Z)X,N\rangle$ by switching $X$ and $Y$.
Thus we have
\[\langle R(X,Y)Z,N\rangle=\langle R(Y,Z)X,N\rangle=\langle R(Z,X)Y,N\rangle.\]
By the first Bianchi identity, the sum of these quantities is zero. So they vanish.
\end{proof}

Let $X$, $Y$, $Z$ be tangent vector fields and $N$ be a normal vector field on a hypersurface $\Gamma$ immersed in $M$.
Furthermore let $\ol\nabla$ be the induced connection and $\ol R$ denote the Riemann curvature operator of $\Gamma$. The  covariant derivative of the shape operator $A$ is defined as
$
(\ol\nabla_XA)(Y):=\ol\nabla_X(A(Y))-A(\ol\nabla_XY).
$
Let $(\cdot)^\top$ denote the tangential component with respect to $\Gamma$, and set
$
(X\wedge Y)Z:=\langle Y, Z\rangle X-\langle X, Z\rangle Y.
$
The Gauss-Codazzi equations \cite{dajczer1990}*{p. 24} for $\Gamma$ are 
\begin{gather}\label{eq:gc}
\ol R(X,Y)Z=(R(X,Y)Z)^\top+(A(X)\wedge A(Y))Z,\\
R(X,Y)N=(\ol\nabla_YA)(X)-(\ol\nabla_XA)(Y).\label{eq:gc2}
\end{gather}
Now we are ready to establish the main result of this section:

\begin{proof}[Proof of Proposition \ref{prop:embedding}]
Let $X$, $Y$, $Z$ be tangent vector fields and $N$ be a normal vector field on  $\Gamma$.
By Lemma  \ref{lem:RXY}, $R(X,Y)N=0$ which yields $(R(X,Y)Z)^\top=R(X,Y)Z$. Furthermore, since $K_{M}\equiv k$ on tangents planes of $\Gamma$, we have $R(X,Y)=k\,X\wedge Y$.
Thus \eqref{eq:gc} and \eqref{eq:gc2} reduce to
\begin{gather*}
\ol R(X,Y)Z=k\, (X\wedge Y)Z+(A(X)\wedge A(Y))Z,\\
(\ol\nabla_Y A)X=(\ol\nabla_X A)Y.
\end{gather*}
These are  the Gauss-Codazzi equations if $\Gamma$ was immersed in $\M^n_k$  \cite{dajczer1990}*{p. 24}. Now the fundamental theorem for hypersurfaces \cite[Thm. 2.1(i)]{dajczer1990} completes the proof.
\end{proof}

\begin{note}\label{note:GC}
The $\C^3$ assumption  in Proposition \ref{prop:embedding} provides the minimum regularity required to express the Gauss-Codazzi equations; however, $\C^2$ or even $\C^{1,1}$ regularity might be enough, where the Gauss-Codazzi equations would hold in an integral or distributional sense. See \cites{hartman-wintner1950, mardare2003} where this approach has been worked out  in $\R^3$.
\end{note}

\section{Schur's Comparison Theorem}\label{sec:schur}

Here we generalize  Schur's comparison  theorem for curves in $\R^n$ \cites{chern1967,sullivan2008}, which is sometimes called the ``bow lemma'' \cite{petrunin-zamora2022}, to  $\CAT^n(k_{\leq 0})$ manifolds. A partial extension of Schur's theorem to $\textbf{H}^n$ was studied by Epstein \cite{epstein1985}, and the polygonal version, known as Cauchy's ``arm lemma'' \cite{aigner-ziegler1999}, holds in $\CAT(k_{\leq 0})$ spaces \cite{akp2019}. We begin by reviewing the basic notions of Alexandrov geometry \cites{bridson-haefliger1999,akp2019b,bbi2001} which we need.

Let $\X  $ be a metric space. The distance between a pair of points $p$, $q\in \X  $ is denoted by $|pq|$ or $|pq|_\X$.
A \emph{curve} is a continuous map  $\gamma\colon [a,b]\to \X  $. We also use $\gamma$ to refer to its image $\gamma([a,b])$.  The \emph{length} of $\gamma$, denoted by $|\gamma|$, is the supremum of $\sum|\gamma(t_i)\gamma(t_{i+1})|$ over all partitions $a=t_0\leq \dots \leq t_N=b$ of $[a,b]$. If $|\gamma|=|\gamma(a)\gamma(b)|$ then $\gamma$ is a \emph{geodesic}. We say $\X  $ is a \emph{geodesic space} if every pair of points $p$, $q\in \X  $ can be joined by a geodesic. If these geodesics are unique (up to reparametrization) they will be denoted by $pq$, and $\X  $ is called a \emph{uniquely geodesic space}. A geodesic space $\X  $ is  $\CAT(k_{\leq 0})$  if every (geodesic) triangle $\Delta$ in $\X  $ is \emph{$k$-thin}, i.e., if $\Delta'\subset\M^2_k$ is a triangle with side lengths equal to those of $\Delta$, then the distance between any pairs of points of $\Delta$ does not exceed that of the corresponding points in $\Delta'$. Every $\CAT(k_{\leq 0})$ space is uniquely geodesic. 
The local convexity assumption on the boundary of a $\CAT^n(k_{\leq 0})$ manifold $M$ ensures that small  triangles in $M$ are $k$-thin \cite{alexander-berg-bishop1993}, or $M$ is locally  $\textup{CAT}(k_{\leq 0})$. Since $M$ is simply connected, it follows from the generalized Cartan-Hadamard theorem  \cites{alexander-bishop1990,bridson-haefliger1999,bbi2001} that $M$ is a $\textup{CAT}(k_{\leq 0})$ space.  Thus  $\CAT^n(k_{\leq 0})$ manifolds are uniquely geodesic. 

A curve $\gamma\colon [a,b]\to \X $ has \emph{unit speed} if $|\gamma|_{[t,s]}|=t-s$ for all $a\leq t\leq s\leq b$. The \emph{chord} of $\gamma$ is the geodesic $\gamma(a)\gamma(b)$. We say $\gamma\colon[a,b]\to \M^2_k$ is  \emph{chord-convex} if $\gamma$  together with its chord forms a \emph{convex curve}, i.e.,  the boundary of a convex body. 

\begin{theorem}[Generalized Schur's Comparison]\label{thm:schur}
Let $\gamma_1\colon [0,\ell]\to\M^2_k$, $\gamma_2\colon [0,\ell]\to M$, where $M$ is a $\CAT^n(k_{\leq 0})$ manifold, be $\C^2$ unit speed curves, and $\kappa_1$, $\kappa_2$ denote their geodesic curvatures  respectively.  Suppose that $\gamma_1$ is chord-convex, and $\kappa_2(t)\leq\kappa_1(t)$ for all $t\in[0,\ell]$. Then $|\gamma_2(0)\gamma_2(\ell)|\geq |\gamma_1(0)\gamma_1(\ell)|$. 
\end{theorem}

 We need the following  well-known result. Let $\gamma\colon [0,\ell]\to \X  $ be a unit speed curve, which is \emph{closed}, i.e., $\gamma(0)=\gamma(\ell)$. Let $\ol\gamma \colon [0,\ell]\to \M^2_k$ be another unit speed curve which bounds a convex body $C$. A \emph{nonexpanding} (or $1$-Lipschitz) map from a subset of a metric space into another is a map which does not increase distances. We say that $\ol\gamma$ \emph{majorizes} $\gamma$ provided that there exists a nonexpanding map $f\colon C\to\X  $ with $f\circ \ol\gamma=\gamma$. The curve  $\ol\gamma$ has also been called an ``unfolding'' \cites{cks2002,ghomi-wenk2021} or ``chord-stretching'' \cites{sallee1973,brooks-strantzen1992} of $\gamma$. We call $f$ the \emph{majorization map}. A curve is \emph{rectifiable} if it has finite length. 

\begin{lemma}[Reshetnyak's Majorization Theorem \cites{akp2019,reshetnyak1968}]\label{lem:reshetnyak}
Every closed rectifiable curve in a $\CAT(k_{\leq 0})$ space is majorized by a closed convex curve in $\M^2_k$.
\end{lemma}

The above result allows us to replace $\gamma_2$ in Theorem \ref{thm:schur} by a curve in $\M^2_k$. The other major component of the proof will be a polygonal approximation. For distinct points $p$, $q\in \M^2_k$, let $\overset{\rightharpoonup}{pq}$ denote the unit tangent vector to $pq$ at $p$ which points towards $q$,  and set
$
\measuredangle(p,o,q):=\cos^{-1}\big(\big\langle \overset{\rightharpoonup}{op}, \overset{\rightharpoonup}{oq}\big\rangle\big)
$
for ordered triples of points.
A curve $\gamma\colon [0,\ell]\to  \M^2_k$ is \emph{polygonal} if there are points $0:=t_0< \dots<t_{N+1}:=\ell$ such that $\gamma|_{[t_{i}, t_{i+1}]}$ is a unit speed geodesic, which is called an \emph{edge} of $\gamma$. Then $\gamma(t_i)$ form \emph{vertices} of $\gamma$ for $1\leq i\leq N$. The \emph{angle} of $\gamma$ at each vertex is defined as 
$
\theta_\gamma(t_i):=\measuredangle\big(\gamma(t_{i-1}),\gamma(t_i),\gamma(t_{i+1})\big).
$
An induction on the number of vertices, as in the proof of Cauchy's  arm lemma  \cites{aigner-ziegler1999,sabitov2004}, shows:

\begin{lemma}\label{prop:polyschur}
Let $\gamma_1$, $\gamma_2\colon [0,\ell]\to\M^2_k$ be chord-convex polygonal curves, with vertices at $t_i\in (0,\ell)$, $i=1,\dots, N$. Suppose that each edge of $\gamma_1$ is equal in length to the corresponding edge of $\gamma_2$, and $\theta_{\gamma_2}(t_i)\geq\theta_{\gamma_1}(t_i)$.  Then $|\gamma_2(0)\gamma_2(\ell)|\geq |\gamma_1(0)\gamma_1(\ell)|$.
\end{lemma}

We assume that all $\C^{m\geq 1}$ curves $\gamma$ are immersed, i.e., $|\gamma'|\neq 0$.
A majorizing curve of a $\C^2$ curve may not be $\C^2$. Thus we consider some generalized notions of geodesic curvature developed by Alexander-Bishop \cite{alexander-bishop1996}. Let $\gamma\colon[0,\ell]\to \X$ be a locally one-to-one unit speed curve. Fix $t\in (0,\ell)$. For $ r<t<s$  close to $t$, let $\Delta(r,s) \subset\M^2_k$ be the triangle with side lengths equal to  the distances between $\gamma(r)$, $\gamma(t)$, and $\gamma(s)$. There exists a unique curve of constant curvature $\alpha(r,s)$ in $\M^2_k$ which circumscribes $\Delta (r,s)$. The upper and lower \emph{osculating curvatures} of $\gamma$ at $t$ are defined  respectively as
$$
\textup{$\ol{osc}$-$\kappa$}(t):=\limsup_{r,s\to t} \alpha(r,s),\quad\quad\text{and}\quad\quad \textup{$\ul{osc}$-$\kappa$}(t):=\liminf_{r,s\to t} \alpha(r,s).
$$
There exists also a curve of constant curvature $\beta(r,s)$ in $\M^2_k$ with a pair of points $p$, $q$ such that the arc length distance $\arc{pq}=r+s$, and the chord distance $|pq|_{\M^2_k}=|\gamma(r)\gamma(s)|_M$. The  upper and lower \emph{chord curvatures} of $\gamma$ at $t$ are defined  respectively as 
$$
\textup{$\ol{chd}$-$\kappa$}(t):=\limsup_{r,s\to t} \beta(r,s),\quad\quad\text{and}\quad\quad \textup{$\ul{chd}$-$\kappa$}(t):=\liminf_{r,s\to t} \beta(r,s).
$$
If $\mathcal{X}$ is a Riemannian manifold and $\gamma$ is $\C^4$, then all these curvatures coincide with the standard geodesic curvature $\kappa(t)$ of $\gamma$ \cite{alexander-bishop1996}. This is also the case  for any $\C^2$ curve in $\M^2_k$. We need the following fact which is a quick consequence of 
\cite[Cor. 3.4]{alexander-bishop1996}:

\begin{lemma}[\cite{alexander-bishop1996}]\label{lem:AB}
Let $\gamma\colon[0,\ell]\to\X$ be a locally one-to-one rectifiable curve, where $\mathcal{X}$ is a $\CAT(k)$ space. Suppose that $\textup{$\ol{chd}$-$\kappa$}(t)\leq f(t)$ for a continuous function $f\colon(0,\ell)\to\R$. Then $\textup{$\ol{osc}$-$\kappa$}(t)\leq f(t)$ as well.
\end{lemma}
\begin{proof}
Let $t_0\in (0,\ell)$. For every $\epsilon>0$, there exits an open neighborhood $U$ of $t_0$ such that $\textup{$\ol{chd}$-$\kappa$}(t)<f(t_0)+\epsilon$ for $t\in U$, since $f$ is continuous. Thus, by \cite[Cor. 3.4]{alexander-bishop1996}, $\textup{$\ol{osc}$-$\kappa$}(t)<f(t_0)+\epsilon$ as well for $t\in U$. In particular $\textup{$\ol{osc}$-$\kappa$}(t_0)<f(t_0)+\epsilon$. So $\textup{$\ol{osc}$-$\kappa$}(t_0)\leq f(t_0)$, which completes the proof.
\end{proof}

Now we establish the main result of this section:

\begin{proof}[Proof of Theorem \ref{thm:schur}]
Join  $\gamma_2$ to its chord to obtain a closed curve. By Lemma \ref{lem:reshetnyak} this curve is majorized by  a  curve in $\M^2_k$. The majorizing curve consists of a chord convex curve, say $\tilde\gamma_2$, and its own chord, which has the same length as the chord of $\gamma_2$.  Note that $\textup{$\ol{chd}$-$\tilde\kappa_2$}\leq \textup{$\ol{chd}$-$\kappa_2$}$  by the majorization property. 

First assume that $\kappa_2<\kappa_1$. Then, after a perturbation, we may assume that $\gamma_2$ is $\C^\infty$, which ensures  that $\textup{$\ol{chd}$-$\kappa_2$}=\kappa_2$.
So $\textup{$\ol{chd}$-$\tilde\kappa_2$}\leq\kappa_2$. Then $\textup{$\ol{osc}$-$\tilde\kappa_2$}\leq\kappa_2<\kappa_1$, by Lemma \ref{lem:AB}.
Replacing $\gamma_2$ by $\tilde\gamma_2$, we  write $\textup{$\ol{osc}$-$\kappa_2$}<\kappa_1$. 
 There exist  oriented polygonal curves $\pi_i^N$  with $N-1$ edges of length $\ell/N$ such that the initial point of $\pi_i^N$ coincides with $\gamma_i(0)$,  the vertices of $\pi_i^N$ lie on $\gamma_i$, and the last vertex of $\pi_i^N$ converges to $\gamma_i(\ell)$ as $N\to\infty$. Since $\textup{$\ol{osc}$-$\kappa_2$}<\kappa_1=\textup{$\ul{osc}$-$\kappa_1$}$,  angles of $\pi_2^N$  will  not be smaller than the corresponding angles of $\pi_1^N$ for large $N$. Thus, by Lemma \ref{prop:polyschur}, the chord of $\pi_2^N$ is not smaller than that of $\pi_1^N$ for large $N$. So letting $N\to\infty$ completes the proof.

Next  consider the case $\kappa_2\leq \kappa_1$. We may assume that $\gamma_1(0)\neq\gamma_1(\ell)$. Let $L\subset\M^2_k$ be the complete geodesic which contains $\gamma_1(0)\gamma_1(\ell)$. If $\gamma_1$ meets $L$ transversely at both ends, then $\gamma_1$ remains chord-convex after a small perturbation. In particular we may replace $\gamma_1$ with the curve $\gamma_1^\epsilon$ in $\M^2_k$ with prescribed curvature $\kappa_1+\epsilon$, for $\epsilon>0$. Then, as discussed above, $|\gamma_2^\epsilon(0)\gamma_2^\epsilon(\ell)|\geq |\gamma_1^\epsilon(0)\gamma_1^\epsilon(\ell)|$ and letting $\epsilon\to 0$ completes the proof. 

So we may assume that $\gamma_1$ is tangent to $L$ at one of its ends, say $\gamma_1(\ell)$, and $\gamma'_1(\ell)$ points towards $\gamma_1(0)$. If $\gamma_1([\ell-{\epsilon}, \ell])\not\subset L$ for small $\epsilon>0$, then $\gamma_1([0, \ell-\epsilon])$  is chord-convex and transversal  to the geodesic through its end points.
So by the last paragraph $|\gamma_2(0)\gamma_2(\ell-\epsilon)|\geq |\gamma_1(0)\gamma_1(\ell-\epsilon)|$ and letting $\epsilon\to 0$ completes the proof. 
Thus we may assume that a segment of $\gamma_1$ near $\ell$ lies on $L$. Let $\ell'\in[0,\ell]$ be 
 the smallest number such that $\gamma_1([\ell',\ell])\subset L$. Then
$
 |\gamma_1(0)\gamma_1(\ell')|\leq |\gamma_2(0)\gamma_2(\ell')|,
$
as we just showed, and since  $\gamma_i$ have unit speed, 
$
|\gamma_2(\ell')\gamma_2(\ell)|\leq \ell-\ell'=|\gamma_1(\ell')\gamma_1(\ell)|.
$
So, since $\gamma_1(\ell)$  lies between $\gamma_1(0)$ and $\gamma_1(\ell')$ on $L$, 
$
|\gamma_1(0)\gamma_1(\ell)| = |\gamma_1(0)\gamma_1(\ell')|-|\gamma_1(\ell')\gamma_1(\ell)|
\leq |\gamma_2(0)\gamma_2(\ell')|-|\gamma_2(\ell')\gamma_2(\ell)| \,\;\leq\;\, |\gamma_2(0)\gamma_2(\ell)|,
$
as desired.
\end{proof}

\begin{note}
The proof of Theorem \ref{thm:schur} shows that we have established something more general: the ambient space $M$ of $\gamma_2$ may be  replaced by any $\CAT(k)$ space, where we relax the condition $\kappa_2(t)\leq\kappa_1(t)$ to $\textup{$\ol{chd}$-$\kappa_2$}(t)\leq\kappa_1(t)$ for $t\in(0,\ell)$.
\end{note}

\begin{note}\label{note:schur}
As is the case in $\R^n$ \cite{sullivan2008}, Theorem \ref{thm:schur} can likely be generalized to $\C^{1,1}$ curves, where the pointwise inequality $\kappa_2\leq\kappa_1$ is replaced by $\int_a^b\kappa_2dt\le\int_a^b\kappa_1dt$ for every subinterval $[a,b]\subset[0,\ell]$.
\end{note}

\section{Proofs of Theorems \ref{thm:main} and \ref{thm:main2}}\label{sec:proof}

Let  $M$ be a $\CAT^n(k_{\leq 0})$ manifold. We need the following special case of a theorem of Lang-Schroeder \cite{lang-schroeder1997}  who generalized Kirszbraun's extension theorem to $\CAT(k)$ spaces; see also \cite{akp2011} \cite[Chp. 10]{akp2019}.

\begin{lemma}[\cite{lang-schroeder1997}]\label{lem:lang-schroeder}
Let  $S\subset\M^n_k$. Then any nonexpanding map $S\to M  $ extends to a nonexpanding map $\M^n_k\to M  $.
\end{lemma}

 Let $\X  $, $\X  '$ be geodesic spaces,  $S\subset \X  $ and $S'\subset \X  '$ be path connected subsets, and 
 $f\colon S\to S'$ be a bijection. We say $f$ is an \emph{extrinsic isometry} provided that $|f(p)f(q)|_{\X'}=|pq|_\X$ for all $p$, $q\in S$. On the other hand,  $f$ is an \emph{(intrinsic) isometry} if it preserves the lengths of curves in $S$. When $S$ and $S'$ are convex,  the two notions coincide.

\begin{lemma}\label{lem:C}
Let  $C\subset \M^n_k$ and $C'\subset  M $  be compact convex bodies. Suppose there exists an extrinsic isometry between boundaries of $C$ and $C'$.  Then $C$ and $C'$ are isometric. 
\end{lemma}
\begin{proof}
Let $\Gamma$ and $\Gamma'$ denote the boundaries of $C$ and $C'$ respectively, and
$f\colon\Gamma\to\Gamma'$ be an extrinsic isometry. By Lemma \ref{lem:lang-schroeder}, $f$ extends to a nonexpanding map $\ol f\colon C\to  M $. We claim that $\ol f$ is an isometry between $C$ and $C'$. 
Let $x_i$, $i=1$, $2$, be distinct points of $C$. Since  $M$ is a $\CAT^n(k_{\leq 0})$ manifold and $C$ is compact, $x_1x_2$ may be extended from each of its end points until it meets $\Gamma$, say at points $y_1$, $y_2$ respectively. 
Let $x_i':=\ol f(x_i)$, $y_i'=\ol f(y_i)$ and $(y_1y_2)':=\ol f(y_1y_2)$.  Then $|y_1'y_2'|\leq |(y_1y_2)'|\leq |y_1y_2|=|y_1'y_2'|$. Thus $|(y_1y_2)'|=|y_1'y_2'|$ which yields that $(y_1y_2)'=y_1'y_2'$. In particular $x_i'$ lies on $y_1'y_2'$. Consequently $|y_1x_i|+|x_iy_2|=|y_1y_2|=|y_1'y_2'|=|y_1' x_i'|+|x_i'y_2'|$. It follows that
$|y_1 x_i|=|y_1' x_i'|$ and $|x_iy_2|=|x_i'y_2'|$. So $|x_1x_2|=|y_1y_2|-|y_1x_1|-|y_2x_2|=|y_1'y_2'|-|y_1'x_1'|-|y_2'x_2'|=|x_1'x_2'|$. Thus $\ol f\colon C\to \ol f(C)$ is an extrinsic isometry. Also since $(y_1y_2)'=y_1'y_2'$ and $C'$ is convex, $\ol f(C)\subset C'$. It remains only to check that $\ol f$ is onto. Given $x'\in C'$, let $y_1'y_2'$ be a geodesic passing through $x'$, with $y_1'$, $y_2'\in\Gamma'$. Let $y_i:=f^{-1}(y_i')$. Then $(y_1y_2)'=y_1'y_2'$ as shown earlier. So $x'\in\ol f(C)$, which completes the proof.
\end{proof}

See \cite[Sec. 2]{croke2004} for results similar to the last lemma. Next we establish a rigidity property of majorizing curves, which is  known in $\R^2$ \cite{brooks-strantzen1992}. A pair of subsets $A$, $B$ of a metric space $\X$ are \emph{congruent} provided that there is an isometry, or \emph{rigid motion}, $f\colon \X\to \X$ with $f(A)=B$.

\begin{lemma}\label{lem:new}
Let $\gamma_1$, $\gamma_2$ be $\C^2$ closed convex curves in $\M^2_k$. Suppose that  $\gamma_1$ majorizes $\gamma_2$. Then $\gamma_1$ and $\gamma_2$ are congruent.
\end{lemma}
\begin{proof}
Let $C_i$ be the convex bodies bounded by $\gamma_i$,  $\gamma_i(t)$ be unit speed parametrizations where $t\in\R/\ell$,  and $f\colon\gamma_1\to\gamma_2$ be the majorization map with $f(\gamma_1(t))=\gamma_2(t)$. By assumption, $|\gamma_1(t)\gamma_1(s)|\geq |\gamma_2(t)\gamma_2(s)|$ for all $t$, $s\in \R/\ell$. Let $\kappa_i(t)$ denote the curvature of $\gamma_i$. Suppose that $\kappa_1(t_0)>\kappa_2(t_0)$, for some $t_0\in\R/\ell$. After a rigid motion  we may assume that $\gamma_1(t_0)=\gamma_2(t_0)=o$, and $\gamma_1$, $\gamma_2$ are tangent to each other at $o$, and lie on the same side a geodesic which passes through $o$. Then there exists a neighborhood $U$ of $o$ in $\gamma_1$ such that $U\setminus\{o\}$ lies in the interior of $C_2$. It follows that $|\gamma_1(t_0-\epsilon)\gamma_1(t_0+\epsilon)|< |\gamma_2(t_0-\epsilon)\gamma_2(t_0+\epsilon)|$, for some $\epsilon>0$, which is a contradiction. Thus $\kappa_1(t)\leq\kappa_2(t)$ for all $t\in\R/\ell$. Furthermore, if $|C_i|$ denote the area of $C_i$, then $|C_2|\leq |C_1|$, since by definition $f$ extends to a nonexpansive map $C_1\to C_2$.
Thus, by Gauss-Bonnet theorem,
$$
2\pi-k|C_1|=\int_{0}^\ell\kappa_1(t)\,dt\leq \int_0^\ell\kappa_2(t)\,dt=2\pi-k|C_2|\leq 2\pi -k |C_1|.
$$
So $\int_{0}^\ell\kappa_1(t)\,dt= \int_0^\ell\kappa_2(t)\,dt$, which yields $\kappa_1\equiv\kappa_2$. Hence $\gamma_1$ and $\gamma_2$ are congruent by the uniqueness of solutions to the geodesic curvature equation.
\end{proof}

Combining the last two observations  with the generalized Schur's comparison theorem and Reshetnyak's majorization theorem, we obtain the following key result. 

\begin{proposition}\label{prop:main}
Let  $C\subset  M $ and $C'\subset \M^n_k$  be compact convex bodies with $\C^2$ boundaries $\Gamma$ and $\Gamma'$ respectively. Suppose that there exists an isometry $\Gamma\to\Gamma'$ which preserves the second fundamental form. Then $C$ and $C'$ are isometric. 
\end{proposition}
\begin{proof}
Let $f$ be the isometry between $\Gamma$, $\Gamma'$, and
for any $x\in\Gamma$ set $x':=f(x)$.
By Lemma  \ref{lem:C} it suffices to show that  for every pair of points $x$, $y\in\Gamma$, 
$
|xy|_M= |x'y'|_{\M^n_k}.
$ 
Let $\Pi$ be a totally geodesic complete surface in $\M^n_k$ containing $x'y'$.  We may assume that $\Pi$ is transversal to $\Gamma'$. So $\gamma':=\Pi\cap\Gamma'$ is a convex curve in $\Pi$. Let $\arc{x'y'}$ be one of the arcs connecting $x'$, $y'$ in $\gamma'$, and $\arc{xy}:=f^{-1}(\arc{x'y'})$ be the corresponding arc in $\gamma:=f^{-1}(\gamma')$. Since $f\colon\Gamma\to\Gamma'$ is an isometry which preserves the second fundamental form, $f\colon \arc{xy}\to\arc{x'y'}$ preserves both the arc length and geodesic curvature of $\arc{xy}$.  Since $\Pi$ is totally geodesic, the geodesic curvature of $\arc{x'y'}$ in $\M^n_k$ is the same as its geodesic curvature in $\Pi$, which is isometric to $\M^2_k$. Thus, by Theorem \ref{thm:schur}, $|xy|_M\geq|x'y'|_\Pi$. Since $\Pi$ is totally geodesic, $|x'y'|_\Pi=|x'y'|_{\M^n_k}$.
So  $|xy|_M\geq |x'y'|_{\M^n_k}$, or $f$ is nonexpanding. To establish the reverse inequality note that
by Reshetnyak's theorem (Lemma \ref{lem:reshetnyak}), there exists a convex curve $\gamma''\subset \Pi$ and a majorization map  $g\colon \gamma''\to \gamma$. Then $f\circ g\colon \gamma''\to\gamma'$ is a majorization map.  So  $\gamma'$ and $\gamma''$ are congruent by Lemma \ref{lem:new}, which yields  $f\circ g$ is an extrinsic isometry. Thus $f|_\gamma$ is noncontracting, i.e., $|xy|_M\leq |x'y'|_{\M^n_k}$. 
\end{proof}

The next observation is implicit in the work of Sacksteder \cite[Sec. 4]{sacksteder1960} for $k=0$, and for $k<0$ is proved similarly via the projective model of the hyperbolic space  \cite{docarmo-warner1970}. A hypersurface $\Gamma$ is \emph{locally convex} if $\ff_{\Gamma}$  is positive semidefinite with respect to a choice of normal vector field. A \emph{line} is a complete geodesic. 

\begin{lemma}[\cite{sacksteder1960}]\label{lem:sacksteder}
Let $\Gamma$ be a complete infinitesimally convex $\C^n$  hypersurface immersed in $\M^n_k$. Then either $\Gamma$ contains a line of $\M^n_k$, or it is locally convex. 
\end{lemma}
\begin{proof}
Replacing $\Gamma$ with its universal cover, we may assume that it is simply connected.
First suppose that $k=0$, or $\M^n_k=\R^n$. Let $X$ be the set of flat points of $\Gamma$, i.e., where $\ff_{\Gamma}$ vanishes, and $X_0$ be a component of $X$. By \cite[Thm. 1]{sacksteder1960} the inclusion map $\Gamma\to\R^{n}$ embeds $X_0$ in a convex subset of a hyperplane $H_0$ (here is where the $\C^n$ regularity assumption is used; see \cite[Lem. 6]{sacksteder1960} and the subsequent remark). Thus $\Gamma\setminus X_0$ is connected unless $X_0$ contains a line, in which case we are done. So we may assume that $\Gamma\setminus X$ is connected. Consequently, we may choose a unit normal vector field $N$ on $\Gamma$, which is the opposite of the mean curvature normal on $\Gamma\setminus X$. Then  $\ff_{\Gamma}$ is positive semidefinite with respect to $N$ as desired.

Next suppose that $k<0$. We may assume that $k=-1$ and identify $\M^n_k$, with the unit ball  $B^n\subset\R^n$ by the Beltrami-Klein projective model of $\textbf{H}^n$. 
Then $\Gamma$ forms an infinitesimally convex hypersurface of $\R^n$ \cite[Sec. 5]{docarmo-warner1970} with Cauchy boundary on $\S^{n-1}$. Note that in the projective model, geodesics are line segments in $\R^n$. Thus, again by the proof of \cite[Thm. 1]{sacksteder1960}, $X_0$ forms a convex subset of $H_0\cap B^n$; see the proof of \cite[Lem. 3]{alexander-ghomi2003} for a concise argument.  So if the closure of $X_0$ intersects $\S^{n-1}$ in more than one point, then $\Gamma$ contains a line and we are done. Otherwise $\Gamma\setminus X$ does not separate $\Gamma$, and the rest of the argument proceeds as in the previous case.
\end{proof}

Now we are ready to prove our main results:
\begin{proof}[Proof of Theorem \ref{thm:main}]
Let $\ol\Gamma$ be the universal Riemannian cover of $\Gamma$. By Proposition \ref{prop:embedding}, $\ol\Gamma$ is isometric to a complete immersed hypersurface $\ol\Gamma'$ in $\M^n_k$ with the same second fundamental form. So $\ol\Gamma'$ is infinitesimally convex.   
Then, by Lemma \ref{lem:sacksteder}, either $\ol\Gamma'$ contains a line in $\M^n_k$ or $\ol\Gamma'$ is locally convex. In the former case, $\ol\Gamma$ must contain a line in $M$, because a line of the ambient space lies on a hypersurface if and only if it is a line of the hypersurface, and the second fundamental form vanishes on tangent vectors of the line. So $\Gamma$ contains a line in $M$. This is a contradiction since $\Gamma$ is compact and lines in $M$ are unbounded. Hence $\ol\Gamma'$ is locally convex, which yields that so is $\ol\Gamma$. Consequently, by Alexander's theorem \cite{alexander1977}, $\ol\Gamma$ is convex. In particular $\ol\Gamma$ is embedded, which yields that  $\ol\Gamma=\Gamma$. So $\Gamma$ is convex. By Proposition \ref{prop:main}, the convex bodies bounded by $\Gamma$ and $\Gamma'$ are isometric, which completes the proof.
\end{proof}

Next, to prove Theorem \ref{thm:main2}, we first record the following basic fact:

\begin{lemma}\label{lem:ball}
If $ M $ is compact, then it is convex and homeomorphic to a ball.
\end{lemma}
\begin{proof}
Let $p_0$, $p_1\in\inte( M )$ be points in the interior of $ M $, and $\gamma\colon[0,1]\to \inte( M )$ be a curve with $\gamma(0)=p_0$, $\gamma(1)=p_1$. Let $\ol t\in[0,1]$ be the supremum of $t\in [0,1]$ such that $p_0\gamma(t)\subset\inte(M)$. If $\ol t\neq 1$, then $p_0\gamma(\ol t)$ must be tangent to $\partial M $; therefore, it lies  in $\partial M $ due to local convexity of $\partial M$  \cite{bishop1974}, which is a contradiction. So $\inte( M )$ is convex, which yields that $M$ is convex. Now the exponential map based at an interior point of $ M $ yields a homeomorphism between $ M $ and a star-shaped domain in $\R^n$.
\end{proof}

Now we establish the intrinsic version of Theorem \ref{thm:main}:

\begin{proof}[Proof of Theorem \ref{thm:main2}]
Let $\Gamma_i$ denote the components of $\Gamma$.
Since $\Gamma_i$ is simply connected,   there exits an isometric embedding $f\colon\Gamma_i\to\Gamma'_i\subset\M^n_k$ preserving $\ff_{\Gamma_i}$, by Proposition \ref{prop:embedding}. By do Carmo-Warner's theorem \cite[Sec. 5]{docarmo-warner1970}, $\Gamma_i'$  is convex, which yields that  $\ff_{\Gamma_i}$ is positive semidefinite with respect to some normal vector field. By assumption, $\ff_{\Gamma_i}$ has a positive eigenvalue at some point with respect to the outward normal $N$. So $\Gamma_i$ must be locally convex with respect to $N$. Hence $M$ is a $\CAT^n(k_{\leq 0})$ manifold. Thus $\Gamma$ is connected by Lemma \ref{lem:ball}, or $\Gamma_i=\Gamma$, and $M$ is a convex body (as a subset of itself).  Let $M'$ be the convex body in $\M^n_k$ bounded by $\Gamma'=f(\Gamma)$. Then $f$ is an isometry between  boundaries of $M$ and $M'$. So $M$ and $M'$ are isometric by Proposition \ref{prop:main}.
\end{proof}

\begin{note}
In the application of Lemma \ref{lem:sacksteder} in  Theorem \ref{thm:main} we could have used the fact that $\Gamma$ is \emph{strictly convex}  at one point, since it is compact. This would quickly resolve the case of $\R^n$ in Lemma \ref{lem:sacksteder}, because it would  force $\Gamma$ to be convex by Sacksteder's theorem \cite{sacksteder1960}; however,  there are complete surfaces in $\mathbf{H}^3$ which are infinitesimally convex, and are strictly convex at one point, but are not convex \cite[p. 84]{spivak:v4}.
\end{note}

\begin{note}
Once convexity of $\Gamma$ and $\Gamma'$ in the above arguments has been established,  one may glue the complement of the convex body bounded by $\Gamma'$ to the convex body bounded by $\Gamma$ to obtain a geodesically complete $\CAT(k)$ space $\mathcal{X}$ \cite{kosovskii2004}; however, $\mathcal{X}$ may not be a smooth Riemannian manifold  a priori. If a gap theorem \cites{schroeder-ziller1990,seshadri2009,greene-wu1982,ballmann-gromov-schroeder}
can be generalized to singular spaces to ensure that $\mathcal{X}$ has constant curvature, it would yield an alternative approach to the results above.
\end{note}

\section{Total Absolute Curvature}\label{sec:higher}

Here we establish an analogue of Corollary \ref{cor:main2} for surfaces of genus $g\geq 1$. Recall that $\Gamma_0$ denotes the boundary of the convex hull of $\Gamma$.

\begin{proposition}\label{prop:high-g}
Let $\Gamma$ be a closed $\C^{1,1}$ surface immersed in a $\CAT^3(k_{\leq 0})$ manifold $M$. Then
\begin{equation}\label{eq:higher}
\tilde{\mathcal{G}}(\Gamma)\geq 4\pi -k|\Gamma_0|,
\end{equation}
with equality only if $K_{ M }\equiv k$ on support planes of $\Gamma_0$, and $GK_\Gamma\geq 0$ everywhere.
\end{proposition}
\begin{proof}
Let $\Gamma_0^\epsilon$ denote the outer parallel surface of $\Gamma_0$ at distance $\epsilon>0$. Then $\Gamma_0^\epsilon$ is $\C^{1,1}$ \cite[Lem. 2.6]{ghomi-spruck2022} and thus by Rademacher's theorem its total curvature $\mathcal{G}(\Gamma_0^\epsilon)$ is well-defined. The total curvature of $\Gamma_0$ is defined as
$$
\mathcal{G}(\Gamma_0)
:=
\lim_{\epsilon\to 0}\mathcal{G}(\Gamma_0^\epsilon).
$$ 
It is known that $\epsilon\mapsto \mathcal{G}(\Gamma_0^\epsilon)$ is a monotone function  which does not increase as $\epsilon\to 0$ \cite[Sec. 6]{ghomi-spruck2022}. 
Furthermore $\mathcal{G}(\Gamma_0^\epsilon)\geq 0$ since $\Gamma_0^\epsilon$ is convex, due to the fact that distance from a convex set in a $\CAT^n(0)$ manifold is a convex function \cite[Cor. 2.5]{bridson-haefliger1999}.
Thus $\mathcal{G}(\Gamma_0)$ exists.
By  \eqref{eq:gauss} and Gauss-Bonnet theorem,
\begin{equation}\label{eq:G}
\mathcal{G}(\Gamma_0^\epsilon)=\int_{\Gamma_0^\epsilon} K_{\Gamma_0^\epsilon}-\int_{\Gamma_0^\epsilon}K_{ M }(T_p\Gamma_0^\epsilon)
\geq
4\pi -k|\Gamma_0^\epsilon|
\geq
4\pi -k|\Gamma_0|.
\end{equation}
Here we have also used the fact that $|\Gamma_0^\epsilon|\geq|\Gamma_0|$, which holds since projection by the nearest point mapping into a convex set is nonexpanding \cite[Cor. 2.5]{bridson-haefliger1999}.
So $\mathcal{G}(\Gamma_0)\geq 4\pi-k|\Gamma_0|$.
Let $\mathcal{G}_+(\Gamma):=\int_{\Gamma_+}GK_\Gamma$, where $\Gamma_+\subset\Gamma$ is the region with $GK_\Gamma\geq 0$. Then \begin{equation}\label{eq:tildeGK}
\tilde{\mathcal{G}}(\Gamma)\geq\mathcal{G}_+(\Gamma)\geq\mathcal{G}(\Gamma\cap\Gamma_0)
=
\mathcal{G}(\Gamma_0)
\geq 
4\pi-k|\Gamma_0|,
\end{equation}
where the middle equality  is due to Kleiner \cite{kleiner1992}, see  \cite[Prop. 6.6]{ghomi-spruck2022}.  
If  equality holds in \eqref{eq:higher}, then equalities hold in \eqref{eq:tildeGK}. In particular $\mathcal{G}(\Gamma_0)
= 4\pi-k|\Gamma_0|$, which yields $\mathcal{G}(\Gamma_0^\epsilon)\to 4\pi-k|\Gamma_0|$, as $\epsilon\to 0$. So \eqref{eq:G} implies that $\int_{\Gamma_0^\epsilon}K_{ M }(T_p\Gamma_0^\epsilon)\to k|\Gamma_0|$. Since $K_ M \leq k$, it follows that
$K_{ M }(T_p\Gamma_0^\epsilon)\to k$. But $T_p\Gamma_0^\epsilon$ converge to support planes of $\Gamma_0$. Consequently, $K_{ M }\equiv k$ on support planes of $\Gamma_0$. Finally, equalities in \eqref{eq:tildeGK} include  $\tilde{\mathcal{G}}(\Gamma)=\mathcal{G}_+(\Gamma)$, which yields $GK_\Gamma\geq 0$. 
\end{proof}

\begin{note}
It is unknown whether closed surfaces  with $GK\geq 0$ in a $\CAT^3(k_{\leq 0})$  manifold are convex \cite[Rem. 4]{alexander1977}; otherwise, Proposition \ref{prop:high-g}  would imply via Theorem \ref{thm:main} that $\Gamma$ bounds a $k$-flat convex body, and solve Gromov's problem in all cases. 
\end{note}

\begin{note}
If Theorem \ref{thm:main} holds for $\C^{1,1}$ hypersurfaces (see Notes \ref{note:GC} and \ref{note:schur}), and one can show that $\Gamma_0$ is $\C^{1,1}$, then Proposition \ref{prop:high-g} solves Gromov's problem in all cases. In $\R^n$ it is already known that the convex hull of a closed $\C^{1,1}$ hypersurface is $\C^{1,1}$ \cite[Note 6.8]{ghomi-spruck2022}. See also \cite{lytchak-petrunin2022,borbely1995} for regularity properties of convex hulls in Riemannian manifolds.
\end{note}

\section*{Acknowledgments}
The author is grateful to Anton Petrunin for important input and feedback during this work.
Thanks also to Stephanie Alexander, Werner Ballmann,  Igor Belegradek, Misha Gromov, Joe Hoisington, and Gil Solanes for useful communications.

\bibliography{references}

\end{document}